\documentclass[11pt]{amsart}
\usepackage{amsmath,amssymb}

\theoremstyle{plain}
\newtheorem{theorem}{Theorem}[section]
\newtheorem{proposition}{Proposition}[section]

\newtheorem{lemma}{Lemma}[section]
\theoremstyle{remark}
\newtheorem*{remark}{Remark}

\newcommand\FF{{\mathbb F}}
\newcommand\cH{{\mathcal H}}
\newcommand\wt{{\rm wt}}

\title{Counting De Bruijn sequences as perturbations of linear recursions} 
\author{Don Coppersmith, Robert C. Rhoades, and Jeffrey M. VanderKam}
\address{Center for Communications Research; Princeton, NJ 08534}
\date{}

\begin{document}

\begin{abstract} 
 Every binary  De~Bruijn sequence of order $n$ satisfies a recursion 
 $0=x_n+x_0+g(x_{n-1},\dots,x_1)$.  Given a function $f$ on
  $n-1$ bits, let $N(f; r)$ be the number of functions generating a
  De~Bruijn sequence of order $n$ which are obtained by changing $r$
  locations in the truth table of $f$. We prove a formula for the
  generating function $\sum_r N(\ell; r) y^r$ when $\ell$ is a linear
  function.

  The proof uses a weighted Matrix Tree Theorem and a
  description of the in-trees (or rooted trees) in the $n$-bit
  De~Bruijn graph as perturbations of the Hamiltonian paths in the same
  graph.
\end{abstract}

\maketitle

\section{ Introduction and Statement of Result}

A (binary) De Bruijn sequence of order $n$ is an infinite 0/1 sequence
with period $2^n$ such that every $n$ long pattern appears exactly
once in a period.  The appearance of De Bruijn sequences can be traced
back to at least 1869, to the invention of a Sanskrit word designed to
help students remember all three-syllable meters \cite[Section
  7.2.1.7]{knuthV4A}.  In the last 100 years De Bruijn sequences have
found application in diverse areas such as robotic vision,
cryptography, DNA sequencing, and even magic \cite[Chapters 2 and
  3,]{DG}. For an overview of the history of De Bruijn sequences the
reader may consult \cite[Section 7.2.1.1]{knuthV4A} and
\cite{fredricksen, ralston, stein}.

Every De Bruijn sequence of order $n$, $\left( \ldots, x_{-1}, x_0,
x_1, x_2, \ldots\right)$, satisfies an $n$-bit recursion
$$x_{i+n} = x_i + g(x_{i+n-1}, \ldots, x_{i+1}),$$ with $g\colon
\FF_2^{n-1} \to \FF_2$.  For example, the two De~Bruijn
sequences of order $3$ are
$$\ldots10111000\ldots \ \ \text{ and } \ \  \ldots11101000\ldots$$
The recursions for these two sequences are
$$x_{i+3} = x_i + x_{i+1} + \overline{x_{i+1}} \  \overline{x_{i+2}}
\ \ \ \text{ and } \ \ \
x_{i+3} = x_i + x_{i+2} + \overline{x_{i+1}} \ \overline{x_{i+2}},
$$ where $\overline{x} = 1+x$ when $x\in \FF_2$.  Each of these
recursions is nearly linear.  For instance, the distance between $x_1$
and the first function is only one, because $x_{1} + \overline{x_{1}}
\ \overline{x_{2}} = x_1$, except when $x_1 = x_2 = 0$.

In 1894 Flye Sainte-Marie \cite{sainte-marie} showed the number of
sequences, up to cyclic equivalence, is $2^{2^{n-1}-n}$. This was
rediscovered by De Bruijn in 1946
\cite{debruijn,debruijn2}\footnote{Also in 1946, the existence of such
  sequences (and generalizations to other alphabets) was rediscovered
  by Good \cite{Good}. More recently, many interesting analogs with
  more general combinatorial structures than binary strings have been
  investigated. See, for instance, \cite{CDG}.}.
This paper
provides a generalization of the 1894 result by counting De Bruijn
sequence recursions which differ from a given linear recursion in
exactly $k$ inputs to the functions.

Let $S(n)$ be the set of functions from $\FF_{2}^{n-1} \to \FF_2$ that
generate a De~Bruijn sequence of order $n$.  For example, $S(3) = \{
x_1 + \overline{x_1}\ \overline{x_2}, x_2
+\overline{x_1}\ \overline{x_2}\}.$ For $f\colon \FF_2^{n-1} \to
\FF_2$, define $S(f; k)$ to be the set of $g \in S(n)$ such that the
weight of $g+f$ is $k$.  The weight of a function is the number of
ones in the truth table. In other words, the weight of $f+g$ is the
number of disagreements between the functions $f$ and $g$.  Moreover,
let $N(f; k) = \left| S(f; k) \right|$ and
\begin{equation}\label{eqn:def_gen_func}
G(f; y) := \sum_k N(f; k) y^k.
\end{equation}
For example, if $n=3$ and $f = x_1$, then $S(x_1; 0) = S(x_1; 2) =
S(x_1; 4) = \emptyset$, $S(x_1; 1) = \{ x_1 + \overline{x_1}
\ \overline{x_2}\}$, and $S(x_1; 3) = \{x_2 +
\overline{x_1}\ \overline{x_2}\}$. Thus, $G(x_1; y) = y + y^3$.

The following notation is used to describe the main theorem of this
paper.  Given $f \colon \FF_2^{n-1} \to \FF_2$, let $C(f)$ denote the
set of sequences, up to cyclic equivalence, satisfying $x_{i+n} = x_i
+ f(x_{i+n-1}, \ldots, x_{i+1})$.  A \emph{necklace} of length $r$ is
an equivalence class of strings of length $r$ consisting of 0s and 1s,
taking all cyclic rotations as equivalent.  A necklace of length $r$
is \emph{primitive} if it is not periodic for any $p<r$.  Thus, each element
of $C(f)$ is represented by a primitive necklace class.  For any ${\bf
  c} \in C(f)$ let $d({\bf c})$ be the number of ones in the primitive
necklace class representing ${\bf c}$.

\begin{theorem}\label{thm:main}
  Let $\ell \colon  \FF_2^{n-1} \to \FF_2$ be a linear function with
  constant term equal to $0$.  Then
  $$G(\ell; y) = \sum_k N(\ell; k) y^k = 2^{-n} \prod_{ {\bf c} \in
    C(\ell) } p_{d({\bf c})}(y),$$
where $p_k(y) = (1+y)^k - (1-y)^k$
for $k>0$ and $p_0(y)=1$.
\end{theorem}

\begin{remark}
 For each linear function $\ell\colon \FF_2^{n-1} \to \FF_2$ this
 theorem gives a refinement of the 1894 result of Flye Sainte-Marie,
 since evaluating the generating function at $y=1$ gives the total
 number of De~Bruijn sequences of order $n$, namely $2^{2^{n-1}-n}$.
\end{remark}

\begin{remark}
  The analogous claim for linear $\ell$ with nonzero
  constant term follows from the fact that
  $G(\ell ; y) = y^{2^{n-1}}
  \cdot G(1+\ell; y^{-1})$.
\end{remark}

\begin{remark}
  Mayhew \cite{mayhew1994, mayhew2000, mayhew2001, mayhew2002},
  Fredricksen \cite{fredricksen}, Hauge and Mykkeltveit \cite{HM1996,
    HM1998} and others have considered the sets $S(0_n; k)$ where
  $0_n$ is the $(n-1)$-bit zero function.  Theorem \ref{thm:main}
  specializes to the following formula for the zero function:
    $$G(0_n; y) = \frac{1}{2^n} \prod_{i} p_i(y)^{e_{n,i}}$$
  where $$e_{n,i} = \sum_{d\mid n} L(d,i)$$ and $L(d,i)$ is the
  number of primitive necklace classes of length $d$ with $i$ ones and
  $d-i$ zeros.
\end{remark}

\begin{remark}
  When a linear function $\ell\colon  \FF_2^{n-1} \to \FF_2$ generates a
  $(2^{n}-1)$-periodic sequence without a run of $n$ zeros
  the theorem gives
 $$G(\ell;y) = \frac{1}{2^n} \left( (1+y)^{2^{n-1}} -
  (1-y)^{2^{n-1}}\right)$$ which was proved previously by Michael
  Fryers in 2015~\cite{fryers}.  Fryers' result can be viewed as a
  generalization of the result of Helleseth and Kl{\o}ve \cite{HK}, which
  computes $N(\ell; 3)$ for any such linear $\ell$.  This case was
  conjectured by the second author upon comparison with random
  permutations \cite{HZ}. Moreover, Fryers' result, combined with
  experimental data for the zero function (see \cite{mayhew1994,
    mayhew2001, mayhew2002}), led to a conjecture for the formula in
  Theorem \ref{thm:main}.
\end{remark}

\begin{remark}
  Recent works have considered modifying linear recursive sequences or
  other understood sequences to produce De Bruijn sequences. See, for
  instance, \cite{LZLH} and \cite{MS}.
\end{remark}

The proof of the theorem has three ingredients.  The first ingredient
is a correspondence between rooted spanning trees in the $n$-bit
De~Bruijn graph and Hamiltonian paths in the same graph.  The basic
construction is in \cite{mowle} and \cite[Chapter VI]{golomb}. We
recall this in Section \ref{sec:in-trees}.  The second ingredient in
our proof is a Matrix Tree Theorem.  Such theorems are a common
ingredient in many of the proofs enumerating De~Bruijn sequences (see,
for instance, \cite{stanley_undergrad}).  Here, we use a weighted
version of the Matrix Tree Theorem, which appears to go back to
Maxwell and Kirchhoff \cite{CK,kirchhoff,maxwell} (see
Section~\ref{sec:matrix_tree_theorem}).  Finally, the computation of
the determinant arising in the Matrix Tree Theorem is established by
moving to the character domain.  The final ingredient results in a
shift from the Fibonacci stepping linear recursion to the
corresponding Galois stepping linear recursion.  This is the only
place in our proof where the linearity of $\ell$ arises. See Section
\ref{sec:galois_stepping} for the details.

\section*{Acknowledgments}
We thank Ron Fertig, Vidya Venkateswaran, and especially Michael
Fryers for helpful conversations.

\section{Hamiltonian Paths and In-Trees}\label{sec:in-trees}
 The $n$-bit De~Bruijn graph, denoted $G_n$, is a 2-in
2-out directed graph with $2^n$ vertices corresponding to elements of
$\FF_2^{n}$ and an edge $x_{n-1}\ldots\allowbreak x_1 x_0 \to x_nx_{n-1}\ldots
x_1$ for all choices of $x_{n}, \ldots, x_1, x_0 \in \FF_2.$ Every
binary De~Bruijn sequence of order $n$ uniquely corresponds to a
Hamiltonian tour through the vertices of $G_n$.  In turn, this tour
gives a unique in-tree with root\break ${\bf 0} = 0\ldots0 \in \FF_2^n$. An
in-tree $T$ of $G_n$ is a subgraph of $G_n$ such that (1) $T$ contains
exactly $2^{n}-1$ edges and (2) every vertex other than ${\bf 0}$ is
connected to ${\bf 0}$ by a directed path in $T$.

For example, the De~Bruijn sequence $\ldots11101000$ of order three 
corresponds to the in-tree
  $$ 100 \rightarrow 010 \rightarrow 101 \rightarrow 110 \rightarrow
111 \rightarrow 011 \rightarrow 001 \rightarrow 000. $$
The 
Hamiltonian tour corresponding to this in-tree is obtained by adding
the edge from $000 \to 100$, thus completing the cycle.  In general,
the in-tree is obtained from the Hamiltonian tour by removing the edge
from ${\bf 0}$ to $10\ldots0$.

The following lemma and theorem show how to obtain all in-trees in
$G_n$ with root ${\bf 0}$ from the De~Bruijn sequences of order $n$.
The approach is contained in \cite{mowle} or \cite[Chatper
  VI]{golomb}, but are proven here to keep the exposition
self-contained.  Let $\cH_n$ be the set of in-trees rooted at ${\bf
  0}$ which are constructed by removing the edge ${\bf 0} \to
10\ldots0$ in a Hamiltonian path of $G_n$.

Given $H\in \cH_n$, an edge $a \to b$ of $G_n$ is \emph{consistent with $H$}
if $H$ visits $a$ before $b$.  For each $H\in \cH_n$ define
$\Lambda(H)$ to be the set of all in-trees of $G_n$ rooted at ${\bf
  0}$ all of whose edges are consistent with $H$.

\begin{theorem}\label{thm:all_intrees}
  Let $\Lambda_n$ be the set of all in-trees rooted at ${\bf 0}$ in
  $G_n$. Then $\Lambda_n = \bigcup_{H \in \cH_n} \Lambda(H)$ and
  $\Lambda(H_1) \cap \Lambda(H_2) = \emptyset$ for all $H_1, H_2 \in
  \cH_n$ whenever $H_1 \ne H_2$.

  Moreover, let
  $$S(H) := \{ Xa \in \FF_2^{n}: Xa \text{\rm\ occurs\ before\ }
  X\overline{a} \text{\rm\ in\ } H \text{\rm\ and\ } X \ne {\bf
    0}\}.$$
Then
  each element of $\Lambda(H)$ is obtained by changing the out-edges of
  a unique subset of vertices in $S(H)$.  Moreover, $\left| S(H)
  \right| = 2^{n-1} -1$ and $\left| \Lambda(H) \right| = 2^{2^{n-1} -
    1}$.
\end{theorem}

Theorem \ref{thm:all_intrees} will be established via a counting
argument.  However, before turning to the proof we discuss the set of
graphs $\Omega(H)$ for $H\in \cH_n$ which are obtained by modifying
the out-edge of any subset of elements of $S(H)$.  Theorem~\ref{thm:all_intrees} claims
that each element of $\Omega(H)$ is an in-tree consistent with
$H$; this is established in the following lemma.

\begin{lemma}\label{lem:Omega(H)}
Let $H \in \cH_n$ and let $S(H)$ and $\Omega(H)$ be defined as above.
Then each $T \in \Omega(H)$ is an in-tree consistent with $H$.
\end{lemma}
\begin{proof}
  Let $T$ be obtained from $H\in \cH_n$ by modifying the out-edges
  from the states in $S \subset S(H)$. It is easy to see that every
  node of $G_n$ has a path in $T$ to ${\bf 0}$ because it is true for
  $H$. Moreover, no loops can exist in $T$. Thus $T$ is an in-tree. 

  To see that $T$ is consistent with $H$, consider an edge in $T$, say
  $Xa\to bX$. If $Xa \not\in S$, then $Xa \to bX$ is in $H$ and is
  thus consistent with $H$.  If $Xa \in S$, then $Xa \to
  \overline{b}X$ and $X\overline{a} \to bX$ are edges in $H$, but $Xa$
  appears before $X\overline{a}$, thus the edge $Xa \to bX$ is
  consistent with $H$.  Therefore, each element of $\Omega(H)$ is an
  in-tree consistent with $H$.
\end{proof}

The following lemmas are useful. 
\begin{lemma}\label{lem:fixed_edge}
  Let $H \in \cH_n$. Suppose $X \ne {\bf 0}$ and $Xa \in S(H)$ with
  $Xa \to cX$ in $H$.  For every $T\in \Omega(H)$ the edge
  $X\overline{a} \to \overline{c}X$ is in $T$ and $H$.
\end{lemma}
\begin{proof}
  Since $Xa \to cX$ is in $H$, so is $X\overline{a} \to
  \overline{c}X$.  Since $Xa \in S(H)$ every element of $T\in
  \Omega(H)$ must have the edge $X\overline{a} \to \overline{c}X$ and
  one of $Xa \to cX$ or $Xa \to \overline{c}X$.
\end{proof}

\begin{lemma}\label{lem:Omega_distinct}
  Let $H_1, H_2 \in \cH_n$.  Then $\Omega(H_1) \cap \Omega(H_2) = \emptyset$. 
\end{lemma}
\begin{proof}
  Suppose $H_1 \ne H_2$ are elements of $\cH_n$. Then
  $H_1 : 10\ldots00 \to \cdots \to {\bf 0}$ and $H_2 : 10\ldots00 \to
  \cdots \to {\bf 0}$.  Let $X \in \FF_2^{n-1}$ be the first $X$ such
  that for some $a$, $Xa \to cX$ is in $H_1$ and
  $Xa \to\overline{c}X$ is in $H_2$.  That is, $X$ is the input to the
  function just before the first time the two cycles diverge.  
  Then $Xa \in S(H_1) \cap S(H_2)$.
  Suppose $T \in \Omega(H_1)$.   Then by Lemma \ref{lem:fixed_edge}, 
  $X\overline{a} \to \overline{c}X$ is in $T$.  If $T$ is also in
  $\Omega(H_2)$ then Lemma \ref{lem:fixed_edge} gives that $X\overline{a}
  \to cX$ must be in $T$, which is a contradiction.
\end{proof}

We now turn to the proof of Theorem \ref{thm:all_intrees}.
\begin{proof}[Proof of Theorem \ref{thm:all_intrees}]
Clearly, $\left| S(H) \right| = 2^{n-1} -1$, because
either $X0$ or $X1$, but not both, are in $S(H)$ for each $X \in
\FF_2^{n-1} \setminus {\bf 0}$.  Therefore, $\left| \Omega(H) \right|
= 2^{2^{n-1}-1}$. Hence, $\left| \Lambda(H) \right| \ge \left|
\Omega(H) \right| = 2^{2^{n-1}-1}.$

By Lemma \ref{lem:Omega(H)}, $\Lambda_n \supseteq \bigcup_{H\in \cH_n} \Omega(H)$.
Moreover, by 
Lemma \ref{lem:Omega_distinct} and the above calculation,  
$$ \left| \Lambda_n \right| \ge \sum_{H \in \cH_n} \left| \Omega(H)
\right | = 2^{2^{n-1} - 1} \cdot \left|\cH_n \right|.$$ It is well
known \cite[Chapter 10]{stanley_undergrad}, that
$\left|\Lambda_n\right| = 2^{2^{n}-n-1}$ and $\left| \cH_{n} \right| =
2^{2^{n-1}-n}$ for all $n\ge 1$.  Therefore, we must have $$\Lambda_n
= \bigcup_{H\in \cH_n} \Omega(H).$$ The proof follows from the fact
that $\Omega(H) \subseteq \Lambda(H)$.
\end{proof}

The following weighted version of the De~Bruijn graph is used in
Section~\ref{sec:matrix_tree_theorem}.

\begin{theorem}
Fix $f\colon \mathbb{F}^{n-1}_2\rightarrow \mathbb{F}_2$.  
  Label the edge $x_{n-1} \ldots\allowbreak x_1 x_0 \to x_n x_{n-1} \ldots x_1$ of
$G_n$ by $1$ if $x_n = x_0 + f(x_{n-1}, \ldots, x_1)$ and by $y$
otherwise.  Denote this weighted graph by $G_{n,f}$. For any in-tree
$T \in \Lambda_n$ define the \emph{weight of $T$}, denoted by
$y^{\wt(f;T)}$, to be the product of the weights of edges.
Then
  $$(1+y)^{2^{n-1} -1} \cdot G(f;y) = \sum_{T \in \Lambda_n}
y^{\wt(f;T)}$$
where $G(f;y)$ is defined in
  \eqref{eqn:def_gen_func}.
\end{theorem}
\begin{proof}
  By Theorem \ref{thm:all_intrees} we have
  $$\sum_{T \in \Lambda_n} y^{\wt(f;T)} = \sum_{H\in \cH_n} \sum_{S
    \subset S(H)} y^{\wt(f;T_{S,H})}$$ where $T_{S,H}$ is the tree in
  $\Lambda(H)$ generated by modifying the out-edges of the set $S$, as
  described in Theorem \ref{thm:all_intrees}.
  Since modifying
  the out-edge from each element of $S(H)$ either increases or decreases
  the weight of the tree by a single $y$, we see that
$\sum_{S \subset S(H)} y^{\wt(f;T_{S,H})} = (1+y)^{2^{n-1}-1} \cdot
  y^{\text{minimum } \wt}$,
where we have used $\left| S(H) \right| =
  2^{n-1}-1$.  Finally, the minimum weight is clearly the number of
  places that the feedback function that generates the De~Bruijn
  sequence represented by $H$ differs from the function $f$. 
\end{proof}

\section{ The Matrix Tree Theorem}\label{sec:matrix_tree_theorem}

The following Matrix Tree Theorem was 
proved by Kirchhoff \cite{kirchhoff} and stated by Maxwell
\cite{maxwell}; see also Chajeken and Kleitman \cite{CK}, and the
references therein.

\begin{theorem}\label{thm:mtt}
  Let $G$ have vertices $v_1,\ldots, v_n$. Suppose that an edge from
  $v_i \to v_j$ is given a weight $-M_{i,j}$, and choose $M_{j,j}$ so
  that $\sum_{i} M_{i,j} = 0$ for all $j$.  Then the sum over all
  in-trees with root $v_1$ of the product of all weights assigned to
  the edges of the in-tree is the determinant of the matrix obtained
  by omitting the row and column corresponding to $v_1$ of the matrix
  $M=\left(M_{i,j}\right)$.
\end{theorem}

In practice, the following lemma is often used with Theorem \ref{thm:mtt}.

\begin{lemma}[\protect{\cite[Lemma 9.9]{stanley_undergrad}}]\label{lem:char_poly}
  Let $M$ be a $p \times p$ matrix such that the sum of the entries in
 every row and column is $0$.  Let $M_0$ be the matrix obtained from $M$
 by removing the first row and first column.  Then the coefficient of
 $z$ in the characteristic polynomial $\det( M - z\cdot I)$ (with $I$
 the identity matrix) of $M$ is equal to $-p \cdot \det(M_0)$.
\end{lemma}

For any $f\colon \FF_2^{n-1} \to \FF_2$ let $G_{n,f}$ be the weighted
De~Bruijn graph defined in Section \ref{sec:in-trees}.  Denote the
weighted adjacency matrix by $W_{f,n}$.  So $W_{f,n}$ is the $2^n
\times 2^n$
matrix with a $1$ in the row and column corresponding to
$x_{n-1}\ldots x_1 x_0 \to x_{n} x_{n-1} \ldots x_1$ if $x_n = x_0 +
f(x_{n-1}, \ldots, x_1)$ and a $y$ otherwise.  For example, with
$f(x_2, x_1) = 0$ the weighted adjacency matrix is 
$$W_{1, 3} = \left(\begin{smallmatrix}
    1 & 0 & 0 & 0 & y & 0 & 0 & 0 \\
    y & 0 & 0 & 0 & 1 & 0 & 0 & 0\\
    0 & 1 & 0 & 0 & 0 & y & 0 & 0 \\
    0 & y & 0 & 0 & 0 & 1 & 0 & 0 \\
    0 & 0 & 1 & 0 & 0 & 0 & y & 0 \\
    0 & 0 & y & 0 & 0 & 0 & 1 & 0 \\
    0 & 0 & 0 & 1 & 0 & 0 & 0 & y \\
    0 & 0 & 0 & y & 0 & 0 & 0 & 1 
\end{smallmatrix}\right).$$

Combining Theorems \ref{thm:all_intrees} and \ref{thm:mtt} and Lemma
\ref{lem:char_poly} one obtains the following: 
\begin{proposition}\label{prop:mtt_evaluation}
  Let $f\colon \FF_2^{n-1} \to \FF_2$.  With $G(f;y)$ defined as in
  \eqref{eqn:def_gen_func} and $W_{f,n}$ the adjacency matrix of the
  weighted De~Bruijn graph,
   $$(1+y)^{2^{n-1} -1} \cdot G(f;y) = \left( \frac{1}{2^n (z -
    (1+y))} \cdot \det(z\cdot I - W_{f,n}) \right) \bigg\vert_{z =
    (1+y)}.$$
\end{proposition}
\begin{proof}
  Applying Theorems \ref{thm:all_intrees} and \ref{thm:mtt} we see
  that $(1+y)^{2^{n-1} -1} \cdot G(f;y)$ is equal to the determinant
  of $(1+y)\cdot I_{2^n} - W_{f,n}$ after deleting the first row and
  column.  To apply Lemma \ref{lem:char_poly} it is sufficient to
  notice that the row and column sums of $(1+y)\cdot I_{2^n} -
  W_{f,n}$ are all zero because every state has an edge into it with
  weight $1$ and an edge into it with weight $y$ as well as an edge
  out of it with weight $1$ and out of it with weight $y$.
\end{proof}

\section{ Proof of Theorem \ref{thm:main}}\label{sec:galois_stepping}

 Let $\ell \colon \FF_2^{n-1} \to \FF_2$ such that $\ell(x_{n-1}\ldots
 x_1) = \sum_{i=1}^{n-1} \ell_{i} \cdot x_i$.  By Proposition
 \ref{prop:mtt_evaluation} the proof of Theorem \ref{thm:main} is
 reduced to that of computing the characteristic polynomial of
 $W_{\ell, n}$, where $W_{\ell,n}$ is the weighted adjacency matrix
 acting on formal linear combinations of elements of $\mathbb{F}^n_2$
 by
$$[x_{n-1}\ldots x_0] \to [x_nx_{n-1} \ldots x_1] + y \cdot
[\overline{x_n}x_{n-1} \ldots x_1]$$
where $x_n = x_0 + \sum_{i=1}^{n-1}
x_i \cdot \ell_i$.

Before giving the proof, we define the Galois and Fibonacci cycles of
a linear recursion.  The proof of Theorem \ref{thm:main} will make use
of the correspondence between these cycles.

The Fibonacci cycles of the linear recursion $x_n= x_0 + \ell(x_{n-1},
\ldots, x_1)$ are defined, as in the introduction, to be the set of
binary sequences, up to cyclic shift, satisfying $x_{i+n} = x_i +
\ell_{n-1} x_{i+n-1} + \cdots + \ell_1 x_{i+1}$.  We remark that these
cycles are in one-to-one correspondence with the Fibonacci cycles of
$x_{i+n} = x_i + \widetilde{\ell}(x_{n-1+i}, \ldots , x_{i+1}) := x_i
+ \ell_1 x_{i+n-1} + \cdots + \ell_{n-1} x_{i+1}$. The correspondence
amounts to reversing the sequences. Moreover, the elements of
$C(\ell)$ are in one-to-one correspondence with the cycles of
$C(\widetilde{\ell})$.

The Galois cycles of the linear recursion $x_n= x_0 + \ell(x_{n-1},
\ldots, x_1)$ are defined to be the set of sequences in $\FF_2^n$, up
to cyclic equivalence, which satisfy the linear recursion
$$\left( \begin{smallmatrix}
  \ell_1 & 1 & 0 & \cdots& 0 & 0 \\
  \ell_2 & 0 & 1 & \cdots& 0 & 0 \\
  &   & \vdots &   &  &\\
  \ell_{n-1} & 0 & 0 & \cdots & 0 & 1 \\
  1 & 0 & 0 & \cdots & 0 & 0 
\end{smallmatrix} \right) {\bf v}_i =
{\bf v}_{i+1}.$$
Projecting the Galois cycles onto any coordinate gives the Fibonacci
cycles of the same linear recursion.  See \cite[Corollary
  3.4]{GK}\footnote{This corollary does not appear in the published
  version of this paper, but is in the version available on the second
  author's website.} or \cite[Chapters 3 and 7]{GKbook}, for example.

\begin{proof}[Proof of Theorem \ref{thm:main}]
For $\alpha = \alpha_{n-1}\ldots\alpha_0 \in \FF_2^n$ define 
the character $\alpha(x) = (-1)^{\sum \alpha_i x_i}$ and the
vector $X_\alpha = \sum_x \alpha(x)[x]$.
Then
\begin{align*}
  W_{\ell,n}(X_\alpha) &= \sum_x \alpha(x) W_{\ell,n}(x)\\
  & =  \sum_{x=x_{n-1}\ldots x_1x_0} (-1)^{\sum \alpha_i x_i} [x_nx_{n-1}\ldots x_1] \\
  &     + y \sum_{x=x_{n-1}\ldots x_1x_0} (-1)^{\sum \alpha_i x_i} [\overline{x_n}x_{n-1}\ldots x_1] \\
  & = (1 + (-1)^{\alpha_0} y) \sum_{x=x_{n-1}\ldots x_1 x_0} \alpha(x) [x_nx_{n-1}\ldots x_1],
\end{align*}
where the last step comes
from changing variables in the second summand.
Next change
variables by setting $z = x_nx_{n-1}\ldots x_1$.
The sum is then
$$   W_{\ell,n}(X_\alpha) = (1 + (-1)^{\beta_{n-1}} y) \sum_{z = x_nx_{n-1}\ldots x_1} \beta(z)[z]$$
where $\beta = \beta_{n-1}\ldots \beta_1 \beta_0$ and
$\beta_i =
\alpha_{i+1} + \ell_{i+1} \alpha_0$ for $i< (n-1)$ and $\beta_{n-1} =
\alpha_0$.
In other words,
$$W_{\ell, n}(X_\alpha) = (1 + (-1)^{\beta_{n-1}} y) X_\beta$$
where
$\beta$ is the character after $\alpha$ in the Galois cycle induced by
$\ell$ on the character space.

Thus $W_{\ell,n}$ has eigenspaces corresponding to the cycles of the
Galois stepping register associated to the linear function $\ell$.
Say that there are $r$ characters in the cycle, of which $k$ have ones
in their $\beta_{n-1}$ position.  The characteristic polynomial of
$W_{\ell,n}$ on this space is then $z^r - (1+y)^{r-k}(1-y)^k$, so the
product of the eigenvalues of $(1+y)I - W_{\ell,n}$ on this space is
$(1+y)^r - (1+y)^{r-k}(1-y)^k = (1+y)^{r-k} p_k(y).$

The proof is finished by noting that the Galois cycles induced by
$\ell$ in the character space have exactly the same sizes as the
cycles induced by $\ell$ in the state spaces, and we can map from one
to the other by taking any linear functional of the one, in particular
by taking $\beta_{n-1}$ (see \cite{GK, GKbook} and the discussion
above).  So the product over all $\ell$-based Galois cycles of this
expression is exactly the same as the product over all $\ell$-based
state cycles of this same expression, where $r$ is again the length of
the cycle and $k$ is the number of ones in the state cycle.

\end{proof}

\begin{remark}
The vectors $X_\alpha$ appeared in Fryers's proof of the case when
$\ell$ generates a cycle of length $2^n-1$ \cite{fryers}.
\end{remark}

\end{document}